\documentclass[11pt,a4paper]{amsart}

\usepackage[margin=25truemm,top=30truemm,bottom=30truemm]{geometry}
\usepackage[dvipdfmx]{graphicx}
\usepackage{enumerate}
\usepackage{amsthm}
\usepackage{amssymb}
\usepackage{tikz-cd}
\usepackage{here}
\usepackage{mathtools}

\newtheorem{Prop}{Proposition}[section]
\newtheorem{Thm}[Prop]{Theorem}

\newtheorem{Rem}[Prop]{Remark}

\theoremstyle{definition}
\newtheorem{Def}[Prop]{Definition}
\newtheorem{Ex}[Prop]{Example}

\makeatletter
\@addtoreset{equation}{section}
\makeatother

\title{Nilpotent Lie algebras obtained by quivers and Ricci solitons}
\author{Fumika Mizoguchi \and Hiroshi Tamaru} 

\date{}
\thanks{This work was supported by JST SPRING, Grant Number JPMJSP2139.  The second author was supported by JSPS KAKENHI Grant Number JP22H01124. This work was partly supported by MEXT Promotion of Distinctive Joint Research Center Program JPMXP0723833165.}

\address[F.~Mizoguchi]{Department of Mathematics, Osaka Metropolitan University, Osaka City, Japan 558-8585} 
\email{sw23882u@st.omu.ac.jp}

\address[H.~Tamaru]{Department of Mathematics, Osaka Metropolitan University, Osaka City, Japan 558-8585} 
\email{tamaru@omu.ac.jp}

\DeclareMathOperator{\id}{id}

\DeclareMathOperator{\g}{\mathfrak{g}}
\DeclareMathOperator{\n}{\mathfrak{n}}
\DeclareMathOperator{\Ric}{\mathrm{Ric}}
\DeclareMathOperator{\ric}{\mathrm{ric}}
\DeclareMathOperator{\Der}{\mathrm{Der}}
\DeclareMathOperator{\Path}{\mathrm{Path}}

\begin{document}
	\maketitle
	\begin{abstract}
		Nilpotent Lie groups with left-invariant metrics provide non-trivial examples of Ricci solitons. One typical example is given by the class of two-step nilpotent Lie algebras obtained from simple directed graphs. 
In this paper, however, we focus on the use of quivers to construct nilpotent Lie algebras. 
A quiver is a directed graph that allows loops and multiple arrows between two vertices. 
Utilizing the concept of paths within quivers, we introduce a method for constructing nilpotent Lie algebras from finite quivers without cycles. 
We prove that for all these Lie algebras, the corresponding simply-connected nilpotent Lie groups admit left-invariant Ricci solitons. The method we introduce constructs a broad family of Ricci soliton nilmanifolds with arbitrarily high degrees of nilpotency.

	\end{abstract}
	\section{Introduction}
		Let $(M, g)$ be a Riemannian manifold and $\ric_g$ be the Ricci curvature tensor. The Riemannian metric $g$ is called a {\it Ricci soliton} if there exist $c\in \mathbb{R}$ and a vector field $X$ on $M$ such that
			\[
				\ric_g=c\cdot g+\mathcal{L}_Xg, 
			\]
			where $\mathcal{L}_X$ is the Lie derivative along $X$. In particular, in the case of $\mathcal{L}_Xg=0$, the metric $g$ is said to be {\it Einstein}.  	
		The investigation of Ricci solitons, possibly with some additional structures, is one of central topics in geometry.  Recently, Ricci solitons on homogeneous manifolds have been studied actively, which we mention here (see also \cite{J1}, \cite{J2}, \cite{L2}). Note that the situation of homogeneous Ricci solitons depends on the signature of $c$. In the case of  $c>0$, the study on homogeneous Ricci solitons can be reduced to the study on homogeneous Einstein manifolds. In fact, a homogeneous Ricci soliton with $c>0$ is isometric to the direct product of a flat manifold and a homogeneous Einstein manifold with positive scalar curvature (\cite{PW}). Homogeneous Ricci solitons with $c=0$ are flat (cf.~\cite{A}). Hence, non-trivial homogeneous Ricci solitons satisfy $c<0$. Recently, the longstanding generalized Alekseevskii conjecture has been proved affirmatively by B\"ohm and Lafuente (\cite{BL}), namely, homogeneous Ricci solitons with $c<0$ are isometric to simply-connected solvable Lie groups with left-invariant metrics. A solvable Lie group with a left-invariant metric is called a {\it solvmanifold}, and a nilpotent Lie group with a left-invariant metric is called a {\it nilmanifold}. The famous structural result by Lauret (\cite{L3}) asserts that there is a correspondence between Ricci soliton solvmanifolds and Ricci soliton nilmanifolds. In fact, all Ricci soliton nilmanifolds can be obtained from Ricci soliton solvmanifolds by taking the nilradicals, and  conversely, all Ricci soliton solvmanifolds can be obtained from Ricci soliton nilmanifolds by certain solvable extensions. This concludes that, in the study on homogeneous Ricci solitons, it is particularly important to study nilmanifolds. 
		
		For Ricci soliton nilmanifolds, the two-step case has been studied actively. One of typical examples is the class of two-step nilpotent Lie groups obtained by simple directed graphs by Dani-Mainkar (\cite{DM}), see also Subsection~\ref{graphs}. Note that the underlying Lie algebra is constructed by edges and vertices. For the simply-connected nilpotent Lie groups obtained by this method, the condition whether they admit left-invariant Ricci solitons or not is known by Lauret and Will (\cite{LW}). On the other hand, there are many unexplained problems about higher step nilpotent Lie groups, and not so many examples are known (see \cite{L2} for a list of examples). Hence it is important to construct tractable examples of higher step nilpotent Lie groups, and to study whether they admit left-invariant Ricci solitons or not. 		 
		In this paper, we construct nilpotent Lie algebras from quivers. A quiver is a directed graph where loops and multiple arrows between two vertices are allowed. Recall that, for the two-step nilpotent Lie algebras obtained by simple directed graphs, the edges and vertices are used for the construction. In our construction, we use the notion of paths. A path in a quiver is a sequence of connected arrows. There is a notion of the path algebra of a quiver (see \cite{ASS}, and also Section~\ref{quivers}). The path algebra is an algebra spanned by all paths whose multiplication is defined as follows: for two paths $x$ and $y$, the product $x\cdot y$ is defined by
		\[
			x\cdot y=
			\begin{cases}
				xy & (\text{if }xy \text{ is a path}), \\
				0 & (\text{otherwise}).
			\end{cases}
		\]
		Our Lie algebras constructed by quivers are defined by the commutator products, that is, the Lie brackets are defined by $[x,y]=x\cdot y-y\cdot x$. 
			
		Here we recall some fundamental notions. A path in a quiver is called a {\it cycle} if its source and target coincide. A quiver is said to be {\it finite} if the set of vertices and the set of arrows are both finite. If $Q$ is a finite quiver without cycles, then the Lie algebra obtained by $Q$ is finite-dimensional and nilpotent. We emphasize that, different from the case of directed graphs, from quivers one can construct nilpotent Lie algebras with arbitrarily high nilpotency steps. 
In this paper, we prove the following theorem. 
		\begin{Thm}
			All simply-connected Lie groups obtained by finite quivers without cycles admit left-invariant Ricci solitons. 
		\end{Thm}
		This constructs a large family of examples of Ricci soliton nilmanifolds with arbitrarily high nilpotency steps. Moreover, the nilpotent Lie groups constructed from quivers would be potentially useful for further studies on nilmanifolds. For example, it would be interesting to study whether there exist other distinguished geometric structures on these Lie groups or not. 
		
		The outline of this paper is as follows. In Section~\ref{preliminaries}, we recall some basic facts on left-invariant Riemannian metrics, such as algebraic Ricci solitons.  In order to prove our theorem, it is enough to show that the underlying nilpotent Lie algebras admit inner products which are algebraic Ricci solitons. We also recall some known examples of nilpotent Lie algebras in relation with Ricci solitons. In Section~\ref{quivers}, we define nilpotent Lie algebras $\n_Q$ obtained by quivers $Q$, and describe some examples of them. In Section~\ref{main_theorem}, we prove our main theorem, by producing  inner products  on $\n_Q$ which are algebraic Ricci solitons. The constructed inner products are preserved by the automorphism groups $\mathrm{Aut}(Q)$ of $Q$, and the sets of all paths $\Path(Q)$ in $Q$ are orthogonal bases. Our proof uses an induction with respect to the number of steps involving these properties. We also note that $\Path(Q)$ is a nice basis in the sense of \cite{N}.

	\section{Preliminaries}\label{preliminaries}
		\subsection{Algebraic Ricci solitons}
			In this subsection, we recall some fundamental facts for a simply-connected Lie group with a left-invariant metric. A Lie algebra $\g$ with an inner product $\langle, \rangle$ is called a {\it metric Lie algebra}, and denoted by $(\g, \langle, \rangle)$. Let $(\g, \langle, \rangle)$ be a metric Lie algebra, and $(G, g)$ be the Lie group with left-invariant metric corresponding to $(\g, \langle, \rangle)$. Curvatures of $(G, g)$ can be determined by the corresponding $(\g, \langle, \rangle)$. 
			\begin{Def}
				Let $(\g, \langle, \rangle)$ be a metric Lie algebra and $X$, $Y\in\g$. The {\it Levi-Civita connection} $\nabla:\g\times\g\to\g$ of $(\g, \langle, \rangle)$ is defined by 
				\[
            				2\langle \nabla_X Y,Z\rangle=\langle[X,Y],Z\rangle+\langle[Z, X],Y\rangle+\langle X,[Z,Y]\rangle. 
            			\]
			\end{Def}
			Note that the Levi-Civita connection can be written as
			\[
				\nabla_X Y=\frac{1}{2}[X,Y]+U(X,Y), 
			\]
			where, the symmetric bilinear form $U:\g\times\g\to\g$ is defined by 
			\[
				2\langle U(X,Y),Z\rangle=\langle[Z,X],Y\rangle+\langle X,[Y,Z]\rangle. 
			\]
			\begin{Def}
				Let $(\g, \langle, \rangle)$ be a metric Lie algebra and $X$, $Y$, $Z \in\g$.  
				\begin{enumerate}[(1)]
					\item The {\it Riemannian curvature} $R: \g \times \g \times \g \to \g$ is defined by 
					\[
						R(X, Y)Z\coloneqq \nabla_X \nabla_Y Z-\nabla_Y \nabla_X Z-\nabla_{[X,Y]} Z.
					\]
					\item The {\it Ricci curvature} $\Ric: \g\to\g$ is defined by 
					\[
						\Ric(X)\coloneqq\sum R(X,e_i)e_i, 
					\]
					where $\{e_i\}$ is an orthonormal basis of $\g$. 
				\end{enumerate}
			\end{Def}
			We recall the definitions of some fundamental notations for Lie algebras, such as solvable and nilpotent Lie algebras. 
			\begin{Def}
				Let $D\g=[\g, \g]$, and $D^k\g=D(D^{k-1}\g)$. Then, the Lie algebra $\g$ is said to be {\it solvable} if there exists $r$ such that $D^r\g=0$. 
			\end{Def}
			\begin{Def}
				Let $C^0\g=\g$, and $C^k\g=[C^{k-1}\g, \g]$. Then, the Lie algebra $\g$ is said to be {\it $m$-step nilpotent} if it satisfies $C^{m-1}\g\neq 0$ and $C^m\g=0$. 
			\end{Def}
			\begin{Def}
				Let $\g$ be a Lie algebra. A linear map $D: \g\to\g$ is called a {\it derivation} if for all $a, b\in\g$, 
				\[
					D([a,b])=[D(a), b]+[a, D(b)]. 
				\]
				The set of all derivations is denoted by $\Der(\g)$. 
			\end{Def}
			In order to prove our main theorem, the notation of algebraic Ricci solitons plays an important role. 
			\begin{Def}
				A metric Lie algebra $(\g, \langle, \rangle)$ is said to be {\it algebraic Ricci soliton} if there exist $c \in \mathbb{R}$ and $D \in\Der(\g)$ such that 
				\[
					\Ric=c\cdot\id+D. 
				\]
			\end{Def}
			In the case of simply-connected nilpotent Lie groups, left-invariant Ricci solitons precisely correspond to algebraic Ricci solitons. 
			\begin{Thm}[Lauret (\cite{L1})]
				Let $(G, g)$ be a simply-connected Lie group with a left-invariant metric, and $(\g, \langle, \rangle)$ be the metric Lie algebra corresponding to $(G, g)$. Then, the following properties hold: 
				\begin{itemize}
					\item if $(\g, \langle, \rangle)$ is algebraic Ricci soliton, then $(G, g)$ is Ricci soliton, 
					\item if $G$ is nilpotent and $(G ,g)$ is Ricci soliton, then $(\g, \langle, \rangle)$ is algebraic Ricci soliton. 
				\end{itemize}
			\end{Thm}
			Recall that it is important to study Ricci soliton solvmanifolds by the generalized Alekseevskii conjecture, proved in \cite{BL}. The famous structural result by Lauret (\cite{L3}) asserts that there is a correspondence between Ricci soliton solvmanifolds and Ricci soliton nilmanifolds.
			\begin{Thm}[Lauret (\cite{L3})]
				The following properties hold: 
				\begin{enumerate}
					\renewcommand{\labelenumi}{(\arabic{enumi})}
					\item Let $(S, g)$ be a Ricci soliton solvmanifold, $(\mathfrak{s}, \langle, \rangle)$ be the corresponding metric Lie algebra, and $\n$ be the nilradical of $\mathfrak{s}$. Then, the simply-connected nilmanifold corresponding to $(\n, \langle, \rangle|_{\n\times\n})$ is Ricci soliton. 
					\item Let $(N, g)$ be a Ricci soliton nilmanifold and $(\n, \langle, \rangle)$ be the corresponding metric Lie algebra. Then there exists a solvable metric Lie algebra $(\mathfrak{s}, \langle, \rangle')$ such that the corresponding simply-connected solvmanifold is Ricci soliton, $\n$ is the nilradical of $\mathfrak{s}$, and $\langle, \rangle'|_{\n\times\n}=\langle, \rangle$. 
				\end{enumerate}
			\end{Thm}
			Finally in this subsection, we recall a formula of the Ricci curvature for the nilpotent case. 
			Let $(\n,\langle,\rangle)$ be a metric nilpotent Lie algebras, and $\{X_1,\ldots,X_n\}$ be an orthonormal basis of $(\n,\langle,\rangle)$. Then the following holds (cf. \cite{L2}, \cite{N}): 
			\begin{equation}\label{Ric_nil}
            			\langle\Ric(X),Y\rangle=-\frac{1}{2}\sum_{i,j}\langle[X,X_i],X_j\rangle\langle[Y,X_i],X_j\rangle+\frac{1}{4}\sum_{i,j}\langle[X_i,X_j],X\rangle\langle[X_i,X_j],Y\rangle. 
            		\end{equation}
		\subsection{Nilpotent Lie algebras obtained by graphs}\label{graphs}
		In this subsection, we recall nilpotent Lie algebras obtained by graphs. A directed simple graph is an ordered pair $(V, E)$ consisting of a set of vertices $V$, and a set of edges $E\subset\{(x, y)\in V\times V\mid x\neq y\}$. 
			For $e=(x, y)\in E$, the vertex $x$ is called the {\it source} of $e$ and $y$ is called the {\it target} of $e$. 
			\begin{Def}[Dani-Mainkar (\cite{DM})]
				Let $(V, E)$ be a directed simple graph. The space $\n$ is defined as an $\mathbb{R}$-vector space having the set $V \cup E$ as basis. The space $\n$ becomes a 2-step nilpotent Lie algebra  with Lie bracket given by 
				\[
                        			[x,y]=
                        			\begin{cases}
                        				e & (x,y\in V,e=(x,y)), \\
                        				-e & (x,y\in V,e=(y,x)), \\
                        				0 & (\textrm{otherwise}).
                        			\end{cases}
            			\]	
			\end{Def}
			\begin{Thm}[Lauret-Will (\cite{LW})]
				The nilpotent Lie algebra obtained by a graph admits algebraic Ricci soliton if and only if the graph is ``positive''. 
			\end{Thm}
			We refer to \cite{LW} the definition of the positivity of a graph. 
		\subsection{Nilpotent parts of parabolic subalgebras}
			In this subsection, we recall the solvable parts and the nilpotent parts of the Langlands decompositions of parabolic subalgebras. For details, we refer to \cite{T} and references therein. 
			
			Let $\g$ be a semi-simple Lie algebra, $\g=\mathfrak{k}\oplus\mathfrak{p}$ be the Cartan decomposition, and $\mathfrak{a}$ be a maximal abelian subspace of $\mathfrak{p}$. In the usual way, $\mathfrak{a}$ defines the restricted root system $\Delta$ of $\g$ with respect to $\mathfrak{a}$. Let $\Lambda$ be a simple root system of $\Delta$. For each proper subset $\Phi\subsetneq\Lambda$, a parabolic subalgebra $\mathfrak{q}_{\Phi}$ of $\g$ is defined. Moreover, $\mathfrak{q}_{\Phi}$ has the Langlands decomposition $\mathfrak{q}_{\Phi}=\mathfrak{m}_{\Phi}\oplus\mathfrak{a}_{\Phi}\oplus\n_{\Phi}$, which is a decomposition into a reductive Lie algebra $\mathfrak{m}_{\Phi}$, an abelian Lie algebra $\mathfrak{a}_{\Phi}$ and a nilpotent Lie algebra $\n_{\Phi}$. The solvable subalgebra $\mathfrak{a}_{\Phi}\oplus\n_{\Phi}$ is called the {\it solvable part}, and the nilpotent subalgebra $\n_{\Phi}$ is called the {\it nilpotent part}. 
			
			Here we illustrate an example. In the case of $\g=\mathfrak{sl}(n, \mathbb{R})$, each parabolic Lie algebra $\mathfrak{q}_{\Phi}$ is given by block upper-triangular matrices up to congruence, and its nilpotent part consists of block strictly upper-triangular matrices.  The following example is the nilpotent part of the parabolic subalgebra given by the $(1, 2, 1, 1)$-decomposition. 
			\[
                        		\left\{\left(
                        		\begin{array}{c|cc|c|c}
                        			0 & x_1 &x_2 & x_3 & x_4 \\ \hline
                        			0 & 0 & 0 & x_5 & x_6 \\
                        			0 & 0 & 0 & x_7 & x_8 \\ \hline
                        			0 & 0 & 0 & 0 & x_9 \\ \hline 
                        			0 & 0 & 0 & 0 & 0 
                        		\end{array}
                        		\right)\mid x_i\in \mathbb{R}\right\}. 
            		\]
			\begin{Thm}[Tamaru (\cite{T})]
				Let $\g$ be a semi-simple Lie algebra, $\mathfrak{q}_{\Phi}$ be a parabolic subalgebra of $\g$, and $\mathfrak{q}_{\Phi}=\mathfrak{m}_{\Phi}\oplus\mathfrak{a}_{\Phi}\oplus\mathfrak{n}_{\Phi}$ be the Langlands decomposition of $\mathfrak{q}_{\Phi}$.  Then, the simply-connected Lie group corresponding to the solvable part $\mathfrak{a}_{\Phi}\oplus\mathfrak{n}_{\Phi}$ admits a left-invariant Einstein metric. In particular, the nilpotent part $\n_{\Phi}$ admits an algebraic Ricci soliton. 
			\end{Thm}
			
			Note that there are some intersections between the nilpotent parts $\n_{\Phi}$ described above, and the nilpotent Lie algebras obtained by quivers. However each one is not a subclass of the other. This will be mentioned in the latter sections. 
	\section{Lie algebras obtained by quivers}\label{quivers}
		In this section, we construct Lie algebras from quivers. In particular, from finite quivers without cycles, one obtains finite-dimensional nilpotent Lie algebras. We refer to \cite{ASS} for quivers and path algebras. 
		\subsection{Preliminaries on quivers}
			In this subsection, we recall some basic notions on quivers. A quiver is a directed graph where loops and multiple arrows between two vertices are allowed. The precise definition is given as follows. 
			\begin{Def}
				A {\it quiver} $Q=(V, E, s, t)$ is a quadruple consisting of two sets $V$, $E$, and two maps $s, t: E\to V$. Elements of $V$ and $E$ are called {\it vertices} and {\it arrows}, respectively. For an arrow $\alpha\in E$, the vertices $s(\alpha)$ and $t(\alpha)$ are called the {\it source} and the {\it target} of $\alpha$, respectively.
			\end{Def}
			Quivers consist of vertices and arrows. Note that the notion of paths is also important, which is defined as follows. 
			\begin{Def}
				Let $Q=(V, E, s, t)$ be a quiver. A {\it path of length $m$} is a sequence $\alpha_1\alpha_2\cdots \alpha_m$, where $\alpha_1, \ldots, \alpha_m\in E$ and they satisfy $t(\alpha_i)=s(\alpha_{i+1})$ for $i=1,\ldots, m-1$.
			\end{Def}
			\begin{Ex}\label{Ex_quiver}
				The following gives a quiver: $V=\{v_1,v_2,v_3,v_4,v_5\},E=\{a,b,c,d,e\}$, $s(a)=v_1$, $s(b)=v_2$, $s(c)=v_3$, $s(d)=v_3$, $s(e)=v_4$, $t(a)=v_2$, $t(b)=v_4$, $t(c)=v_4$, $t(d)=v_4$, $t(e)=v_5$. It is illustrated as in Figure \ref{quiver}. Paths of this quiver are $a$, $b$, $c$, $d$, $e$, $ab$, $be$, $ce$, $de$, and $abe$. 
				For a path $x=\alpha_1\alpha_2\cdots \alpha_m$ in a quiver $Q$, the vertices $s(\alpha_1)$ and $t(\alpha_m)$ are also called the {\it source} and the {\it target} of $x$, and denoted by $s(x)$ and $t(x)$ respectively.  In this way, the maps $s, t: E\to V$ can be extended to $s, t: \Path(Q)\to V$, where $\Path(Q)$ is the set of all paths in $Q$.   
			\end{Ex}
			\begin{figure}[H]
				\begin{center}
					\includegraphics[scale=0.3]{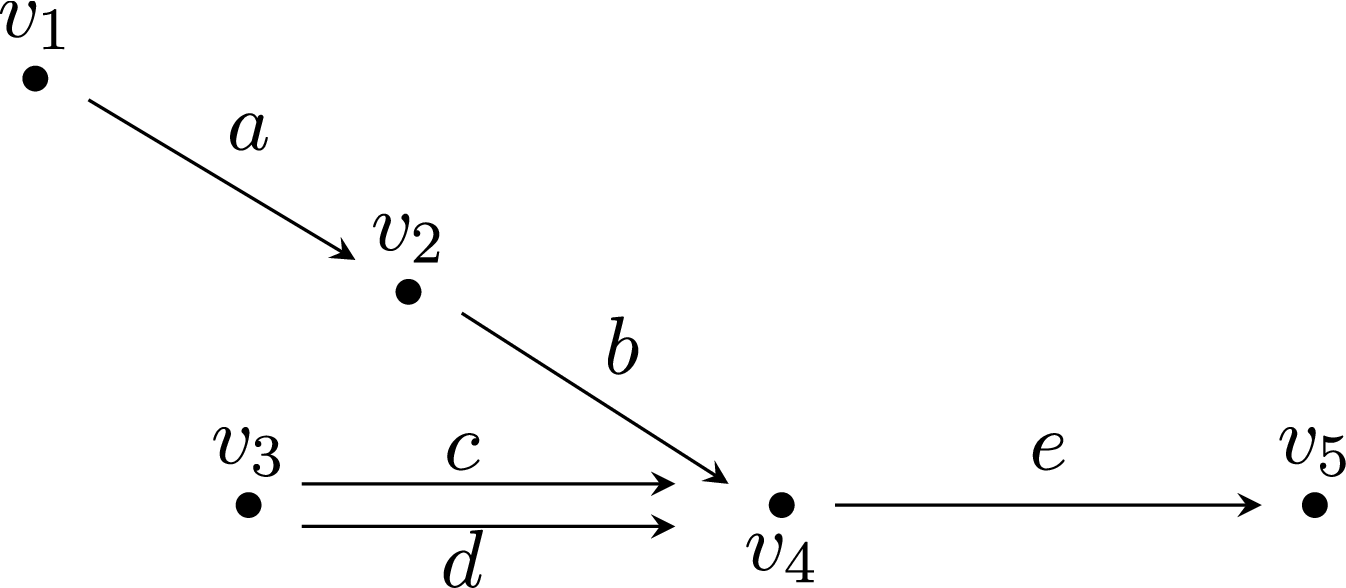}
					\caption{The quiver in Example \ref{Ex_quiver}}
					\label{quiver}
				\end{center}
			\end{figure}
			A path whose source and target coincide is called a {\it cycle}. A quiver is said to be {\it finite} if the set of vertices and the set of arrows are both finite. 
			\begin{Def}
				Let $Q=(V, E, s, t)$ be a quiver. A map $f: E\to E$ is called an {\it automorphism} if the following conditions hold:
				\begin{itemize}
					\item $f$ is bijective,
					\item $t(x)=s(y)$ is equivalent to $t(f(x))=s(f(y))$, for any $x$, $y$ $\in E$. 
				\end{itemize}
			\end{Def} 
			The automorphism group of a quiver $Q$ is defined by 
			\[
				\mathrm{Aut}(Q)\coloneqq\{f: E\to E\mid f:\text{automorphism}\}. 
			\]
			Note that $\mathrm{Aut}(Q)$ acts naturally on $\Path(Q)$. In fact, for a path $x=\alpha_1\alpha_2\cdots \alpha_m$ in a quiver $Q$, the action of $f$ $\in\mathrm{Aut}(Q)$ is given by 
			\[
				f(x)\coloneqq f(\alpha_1)f(\alpha_2)\cdots f(\alpha_m).
			\]
		\subsection{Lie algebras obtained by quivers}
			In this subsection, we define Lie algebras obtained by quivers. Also we introduce some properties of these Lie algebras. Recall that $\Path(Q)$ is the set of all paths in $Q$. 
			\begin{Def}
				Let $Q=(V, E, s, t)$ be a quiver. The {\it path algebra} $\n_Q$ is defined as an $\mathbb{R}$-vector space having $\Path(Q)$ as basis with the following product: it is bilinear, and the product between two paths $x$ and $y$ is defined by the concatenation, that is 
				\[
            				x\cdot y=
            				\begin{cases}
            					xy & (\text{if }xy \text{ is a path}), \\
            					0 & (\text{otherwise}).
            				\end{cases}
            			\]
				The space $\n_Q$ equipped with the bracket product $[x, y]\coloneqq x\cdot y-y\cdot x$ is called the {\it Lie algebra obtained by a quiver}. 
			\end{Def}
			For the path algebras $(\n_Q, \cdot)$, we refer to \cite{ASS}. Note that path algebras usually contain vertices, as paths of length 0. However, we have to note that the Lie algebras $(\n_Q, [,])$ do not contain vertices, that is, we only consider paths of length $\geq1$. 

				
			\begin{Prop}\label{length}
				Let $Q=(V, E, s, t)$ be a finite quiver without cycles. Then, the following properties hold:  
				\begin{enumerate}[(1)]
					\item for a path $\alpha_1\alpha_2\cdots\alpha_r\in\Path(Q)$ with $\alpha_1, \alpha_2, \ldots, \alpha_r\in E$, each edge can appear at  most once. That is, it satisfies $\alpha_i\neq\alpha_j$ if $i\neq j$.  
					\item there exists the maximum of the lengths of paths, which is called the {\it length of a quiver}.
				\end{enumerate} 
			\end{Prop} 
			\begin{proof}
				First of all, we prove (1). Suppose that there exist $i<j$ such that $\alpha_i=\alpha_j$.  Then a path $\alpha_i\cdots\alpha_{j-1}$ is a cycle, since it satisfies $t(\alpha_i\cdots\alpha_{j-1})=t(\alpha_{j-1})=s(\alpha_j)=s(\alpha_i)=s(\alpha_i\cdots\alpha_{j-1})$. Since $Q$ has no cycles by assumption, this is a contradiction, which proves (1). The assertion (2) easily follows from (1), since $\# E<\infty$. 
			\end{proof}
			We now show that the Lie algebras obtained by finite quivers without cycles are nilpotent. Also, the number of steps of $\n_Q$ coincides with  the length of $Q$. 
			\begin{Prop}\label{step}
				Let $Q$ be a finite quiver without cycles. Then the following properties hold:
				\begin{enumerate}[(1)]
					\item for all $x$, $y\in\Path(Q)$, it satisfies $[x, y] \in \{0\}\cup (\pm\Path(Q))$, 
					\item the subspaces $\n_Q^i$ of $\n_Q$ spanned by paths of length $i$ satisfy $[\n_Q^i, \n_Q^j]=\n_Q^{i+j}$, 
					\item if $Q$ is a quiver of length $m$, then the Lie algebra $\n_Q$ is an $m$-step nilpotent Lie algebra. 
				\end{enumerate}
			\end{Prop}
			\begin{proof}
				First of all, we prove (1). Take any $x, y\in\Path(Q)$. By definition, it satisfies $[x, y]=x\cdot y-y\cdot x$. It is enough to show that $x\cdot y=0$ or $y\cdot x=0$. Assume that both are nonzero. Since $xy$ and $yx$ are paths, $xyx$ is a path. This contradicts to Proposition \ref{length} (1), which proves (1). 
				
				We next prove (2). Let $x$ and $y$ be paths of length $i$ and $j$, respectively. It follows from (1) that the bracket $[x, y]$ is $0$, $xy$ or $-yx$. Therefore, it satisfies $[x, y]\in\n_Q^{i+j}$. This proves $(\subset)$. In order to show the converse inclusion, let $z=\alpha_1\alpha_2\cdots\alpha_i\beta_1\beta_2\cdots\beta_j$ be a path of length $i+j$, where $\alpha_1, \ldots, \alpha_i, \beta_1, \ldots, \beta_j\in E$. Set $x=\alpha_1\alpha_2\cdots\alpha_i$ and $y=\beta_1\beta_2\cdots\beta_j$. Then $yx$ is not a path by (1), and hence one has 
				\[
					z=xy=[x,y]\in[\n_Q^i, \n_Q^j]. 
				\]
				This completes the proof of the assertion (2).  
				
				Finally, the assertion (3) follows from (2).  Recall that the descending central series is defined by $C^k\n_Q=[C^{k-1}\n_Q, \n_Q]$ and $C^0\n_Q=\n_Q$. It follows from (2) that  $C^{m-1}\n_Q=\n_Q^m\neq0$ and $C^m\n_Q=[C^{m-1}\n_Q, \n_Q]=[\n_Q^m, \n_Q]=0$. This completes the proof. 
			\end{proof}
			
		\subsection{Examples}
			In this subsection, we describe some examples of nilpotent Lie algebras obtained by quivers. 
			\begin{Ex}\label{Ex1}
				Consider the following quiver $Q$ in Figure \ref{example1}: 
				\noindent Then the Lie algebra $\n_Q$ is given by 
				\[
					\n_Q=\mathrm{span}\{a, b,ab\}. 
				\]
				The Lie bracket is as follows:
				\begin{itemize}
					\item $[a,b]=a\cdot b-b\cdot a=ab$, 
            				\item $[a,ab]=0$, 
            				\item $[b,ab]=0$. 
				\end{itemize}
				Hence, the Lie algebra $\n_Q$ is isomorphic to the 3-dimensional Heisenberg algebra $\mathfrak{h}^3$. 
			\end{Ex}
			\begin{figure}[H]
				\begin{center}
					\includegraphics[scale=0.3]{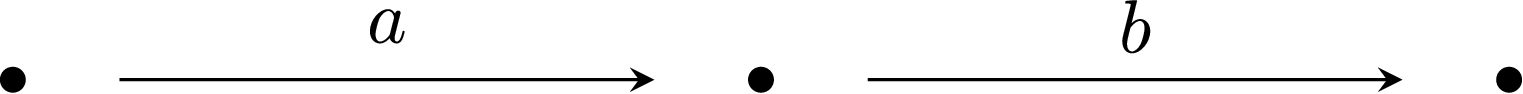}
					\caption{The quiver in Example \ref{Ex1}}
					\label{example1}
				\end{center}
			\end{figure}
			\begin{Ex}\label{Ex2}
				Consider the following quiver $Q$ in Figure \ref{example2}: 
				\noindent Then the Lie algebra $\n_Q$ is given by 
            			\[
            				\n_Q=\mathrm{span}\{a,b,c,d,ac,bc,cd,acd,bcd\}. 
            			\]
            			The Lie algebra $\n_Q$ is isomorphic to: 
            			\[
                                		\left\{\left(
                                		\begin{array}{cc|c|c|c}
                                			0 & 0 & x_1 & x_5 & x_8\\
                                			0 & 0 & x_2 & x_6 & x_9\\ \hline
                                			0 & 0 & 0 & x_3  & x_7\\ \hline
                                			0 & 0 & 0 & 0 & x_4 \\ \hline
                                			0 & 0 & 0 & 0 & 0 \\
                                		\end{array}\right)\mid x_i\in\mathbb{R}\right\}.
                                	\]	
			\end{Ex}
			\begin{figure}[H]
				\begin{center}
					\includegraphics[scale=0.3]{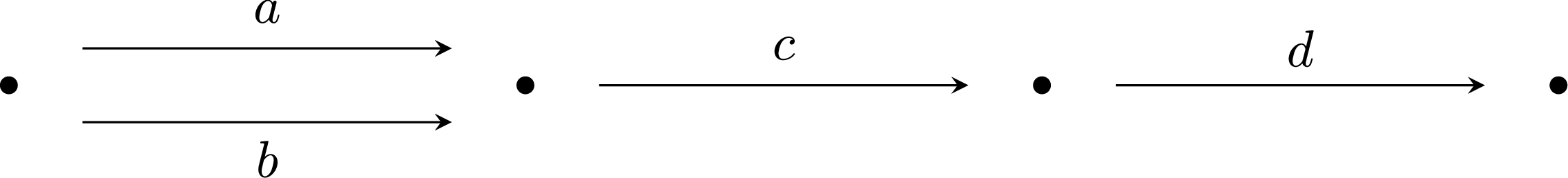}
					\caption{The quiver in Example \ref{Ex2}}
					\label{example2}
				\end{center}
			\end{figure}
			Examples \ref{Ex1} and \ref{Ex2} are obtained by block decomposition of matrices. Therefore, they admit algebraic Ricci solitons. 
			\begin{Rem}\label{Rem1}
				The Lie algebra obtained by the following quiver in Figure \ref{remark1} is the same as the Lie algebra of example \ref{Ex2}. 
			\end{Rem}
			\begin{figure}[h]
				\begin{center}
					\includegraphics[scale=0.3]{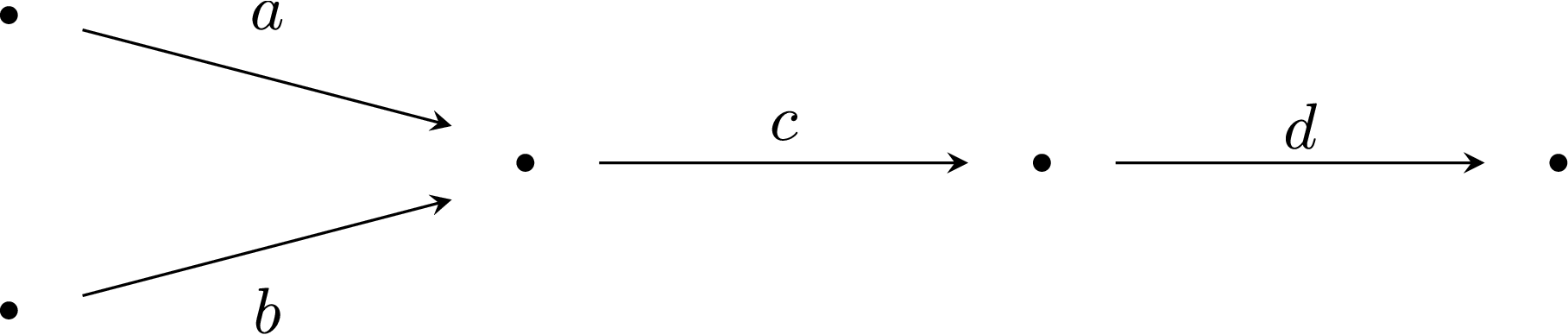}
					\caption{The quiver in Remark \ref{Rem1}}
					\label{remark1}
				\end{center}
			\end{figure}
			\begin{Rem}~
				\begin{itemize}
					\item There exist Lie algebras obtained by block decompositions which cannot be obtained by quivers.  For example, the following $5$-dimensional Heisenberg Lie algebra $\mathfrak{h}^5$ cannot be obtained by quivers: 
					\[
                					\left\{\left(\begin{array}{c|cc|c}
                                			0 & x_1 & x_2 & x_5 \\ \hline
                                			0 & 0 & 0 & x_3 \\
                                			0 & 0 & 0 & x_4 \\ \hline
                                			0 & 0 & 0 & 0 \\
                                			\end{array}\right)\mid x_i\in\mathbb{R}\right\}. 
            				\]
					If this is obtained by a quiver $Q_1$, then $Q_1$ has four paths of length $1$ and one path of length $2$ as in Figure \ref{remark2-2}. Then, $\n_{Q_1}$ is isomorphic to $\mathfrak{h}^3\oplus\mathbb{R}^2$, but not $\mathfrak{h}^5$. 
					\vspace{-5mm}
					\begin{figure}[H]
            					\begin{center}
            						\includegraphics[scale=0.25]{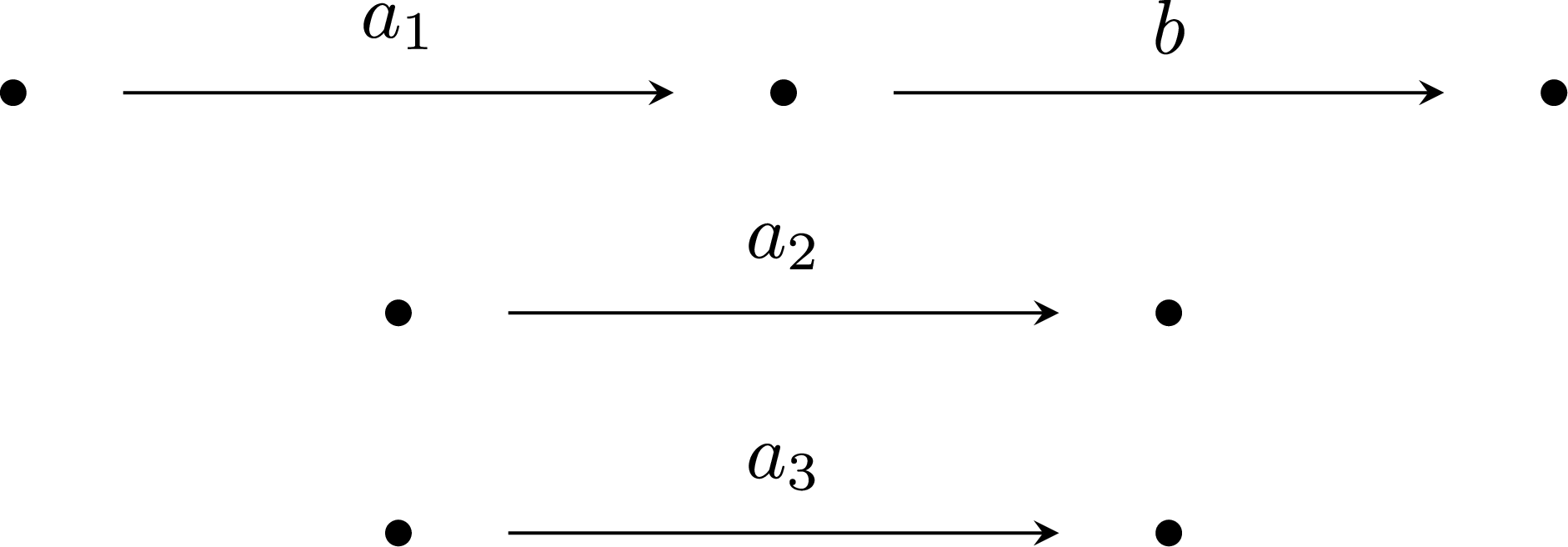}
							\caption{Quiver $Q_1$}
							\label{remark2-2}
            					\end{center}
            				\end{figure}
					\vspace{-5mm}
					\item There exist Lie algebras obtained by quivers which can not be obtained as nilpotent parts of parabolic subalgebras. The Lie algebra obtained by the quiver in Figure \ref{remark2} is an example: 
            				\begin{figure}[H]
            					\begin{center}
            						\includegraphics[scale=0.3]{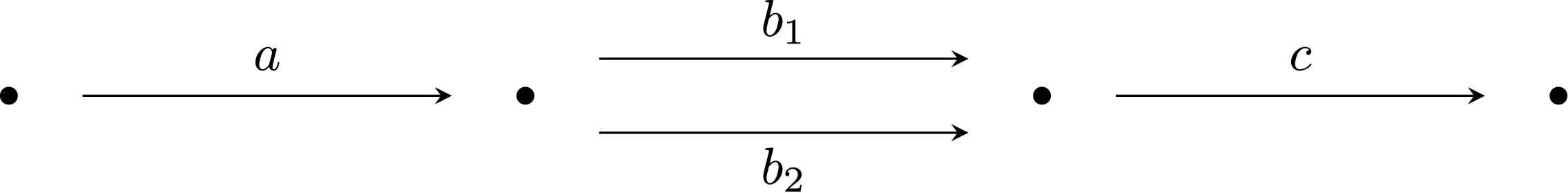}
							\caption{Quiver $Q_2$}
							\label{remark2}
            					\end{center}
            				\end{figure}
					This nilpotent Lie algebra $\n_{Q_2}=\n_{Q_2}^1\oplus\n_{Q_2}^2\oplus\n_{Q_2}^3$ satisfies $\mathrm{dim}\n_{Q_2}^1=4$, $\mathrm{dim}\n_{Q_2}^2=4$ and $\mathrm{dim}\n_{Q_2}^3=2$, where $\n_{Q_2}^i$ is the subspace spanned by paths of length $i$.  By looking at the list of parabolic subalgebras, no nilpotent parts have such dimensions. 
				\end{itemize}
			\end{Rem}
	\section{Main theorem}\label{main_theorem}
		In this section, we show the main theorem. For the nilpotent Lie algebras $\n_Q$, we construct inner products, which are algebraic Ricci solitons, inductively. 
		\subsection{Nice basis}
			In the section, we recall the definition of a nice basis, and prove that $\n_Q$ has a nice basis. 
			\begin{Def}[Nikolayevsky (\cite{N})]\label{nice_def}
				Let $\{X_1,\ldots,X_n\}$ be a basis of a nilpotent Lie algebra $\n$, with $[X_i,X_j]=\sum_{k=1}^nc_{ij}^kX_k$. Then the basis $\{X_1,\ldots,X_n\}$ is said to be {\it nice} if 
				\begin{enumerate}[(1)]
					\item for any $i$ and $j$, there exists at most one $k$ such that $c^k_{ij}\neq 0$, 
					\item for any $i$ and $k$, there exists at most one $j$ such that $c^k_{ij}\neq 0$. 
				\end{enumerate}
			\end{Def}
			If an orthonormal basis $\{X_1, \ldots, X_n\}$ is nice, then the Ricci operator is diagonal with respect to this basis (\cite{N}), that is, 
			\[
				\langle\Ric(X_i),X_j\rangle=0~~~(i\neq j).
			\]
			Therefore, we have the following by Equation \ref{Ric_nil}. 
			\begin{Prop}
				Let $(\n,\langle,\rangle)$ be a metric nilpotent Lie algebra, and $\{X_1,\ldots,X_n\}$ be an orthonormal nice basis of $(\n,\langle,\rangle)$. Then the following equation holds: 
            			\begin{equation}\label{Ric_nice}
                        			\Ric(X_k)=r_k X_k, ~~~~r_k=-\frac{1}{2}\sum_{i,j}\langle[X_k,X_i],X_j\rangle^2+\frac{1}{2}\sum_{i<j}\langle[X_i,X_j],X_k\rangle^2. 
                        		\end{equation}
			\end{Prop}
			We now prove that $\Path(Q)$ is a nice basis of $\n_Q$. This is useful for the later calculations, and also indicates that $\n_Q$ has a relatively simple structure.  
			\begin{Prop}
				Let $Q$ be a finite quiver without cycles, $\Path(Q)$ be the set of all paths in $Q$, and $\n_Q$ be the nilpotent Lie algebra obtained by $Q$. Then, $\Path(Q)$ is a nice basis of $\n_Q$. 
			\end{Prop}
			\begin{proof}
				It is easy to see that $\Path(Q)$ is a basis of $\n_Q$. Let us denote by 
				\[
					\Path(Q)=\{x_i \mid 1\leq i \leq n \},~~[x_i, x_j]=\sum_{k=1}^nc_{ij}^kx_k
				\]
				Take $i$, $j$, then $[x_i, x_j]\in\{0\}\cup\pm\Path(Q)$ from Proposition \ref{step}. Therefore, it satisfies Definition \ref{nice_def} (1). 
				
				It remains to show the condition (2). Take $i$ and $k$, and assume that there exist $j_1$, $j_2$ such that $c_{ij_1}^k\neq 0 \neq c_{ij_2}^k$. By the above argument, it satisfies $[x_i, x_{j_1}]=\pm x_k$ and $[x_i, x_{j_2}]=\pm x_k$. One can complete the proof by studying the four cases. The first case is $[x_i, x_{j_1}]=x_k$ and $[x_i, x_{j_2}]=-x_k$, that is, 
				\[
					x_ix_{j_1}=x_k=x_{j_2}x_i. 
				\]
				In this case, one can show that $s(x_{j_1})=t(x_i)=t(x_k)=t(x_{j_1})$. This means that $x_{j_1}$ is a cycle, which is a contradiction. The second case is $[x_i, x_{j_1}]=x_k$ and $[x_i, x_{j_2}]=x_k$, that is, 
				\[
					x_ix_{j_1}=x_k=x_ix_{j_2}. 
				\]
				This yields that $x_{j_1}=x_{j_2}$, and hence the condition (2) is satisfied in this case. One can show the remaining two cases similarly. 
			\end{proof}
		\subsection{Specific subalgebras}\label{subalgebra}
			Let $Q=(V, E, s, t)$ be a quiver of length $m$, and $\Path(Q)$ be the set of all paths in $Q$. In this subsection, we study a specific subalgebra of $\n_Q$. First of all, we define 
			\[
				S\coloneqq\{a\in E\mid \exists x\in\Path(Q) \text{ s.t. }ax \text{ is a path of length }m\}, 
			\]
			which is called the {\it starting set}. Then we consider a new quiver $Q'=(V', E', s', t')$ defined by 
			 \[
				V'=V,~~E'=(E\setminus S)\cup\{ab\in\Path(Q) \mid a\in S, b\in E\},~~s'=s|_{E'},~~t'=t|_{E'}. 
			\] 
			By definition, it is easy to see that $\Path(Q)=S\sqcup\Path(Q')$. In the following, we gives an easy example of $Q'$. 
			\begin{Ex}\label{Ex_Q}
				Let $Q$ be the quiver in Figure \ref{example_Q}. By definition, one has $S=\{a_1, a_2\}$. Then the quiver $Q'$ becomes the one in Figure \ref{example_Q2}. By sorting out the figure, the nilpotent Lie algebra obtained by $Q'$ is isomorphic to the nilpotent Lie algebra obtained by Figure \ref{example_Q3}. 
			\end{Ex}
			\begin{figure}[H]
				\begin{center}
					\includegraphics[scale=0.3]{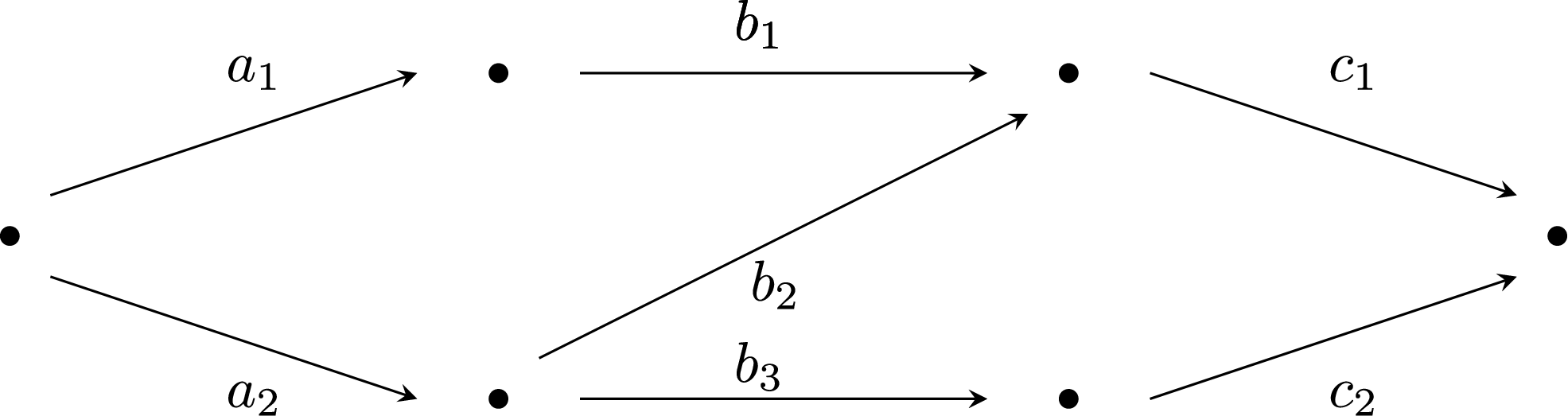}
					\caption{The quiver $Q$ in Example \ref{Ex_Q}}
					\label{example_Q}
				\end{center}
			\end{figure}
			\begin{figure}[H]
        				\begin{center}
        					\includegraphics[scale=0.3]{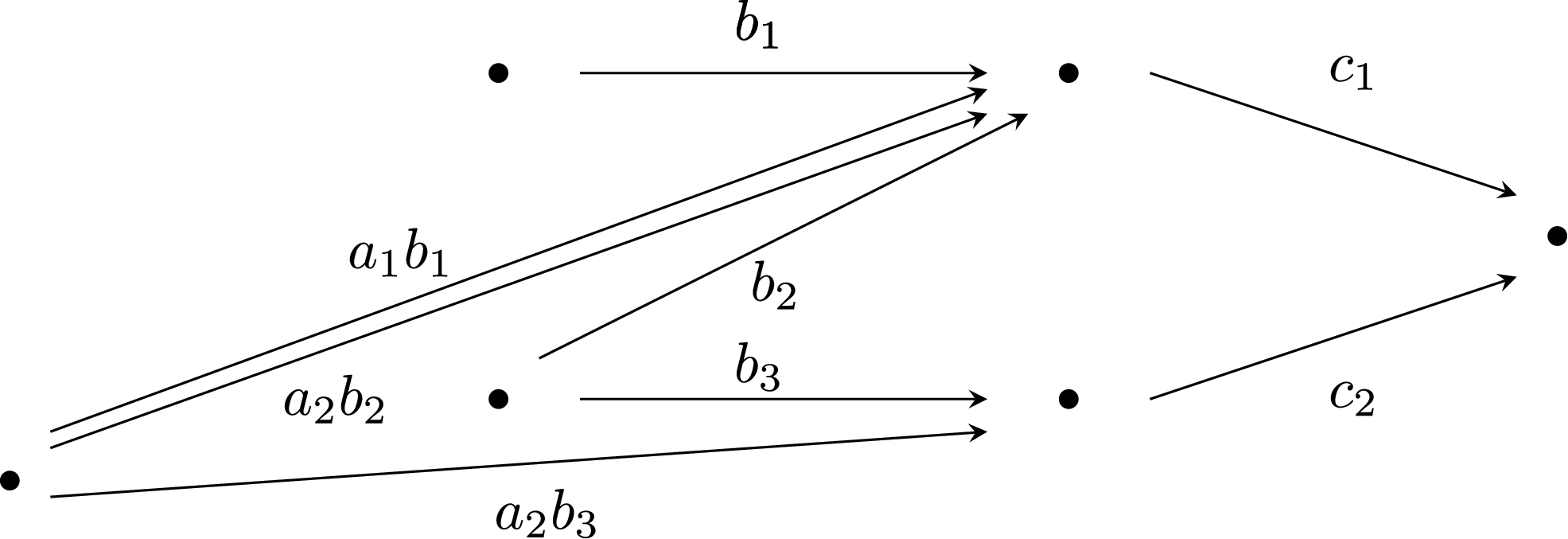}
					\caption{The quiver $Q'$ in Example \ref{Ex_Q}}
					\label{example_Q2}
        				\end{center}
        			\end{figure}
			\begin{figure}[H]
        					\begin{center}
        						\includegraphics[scale=0.3]{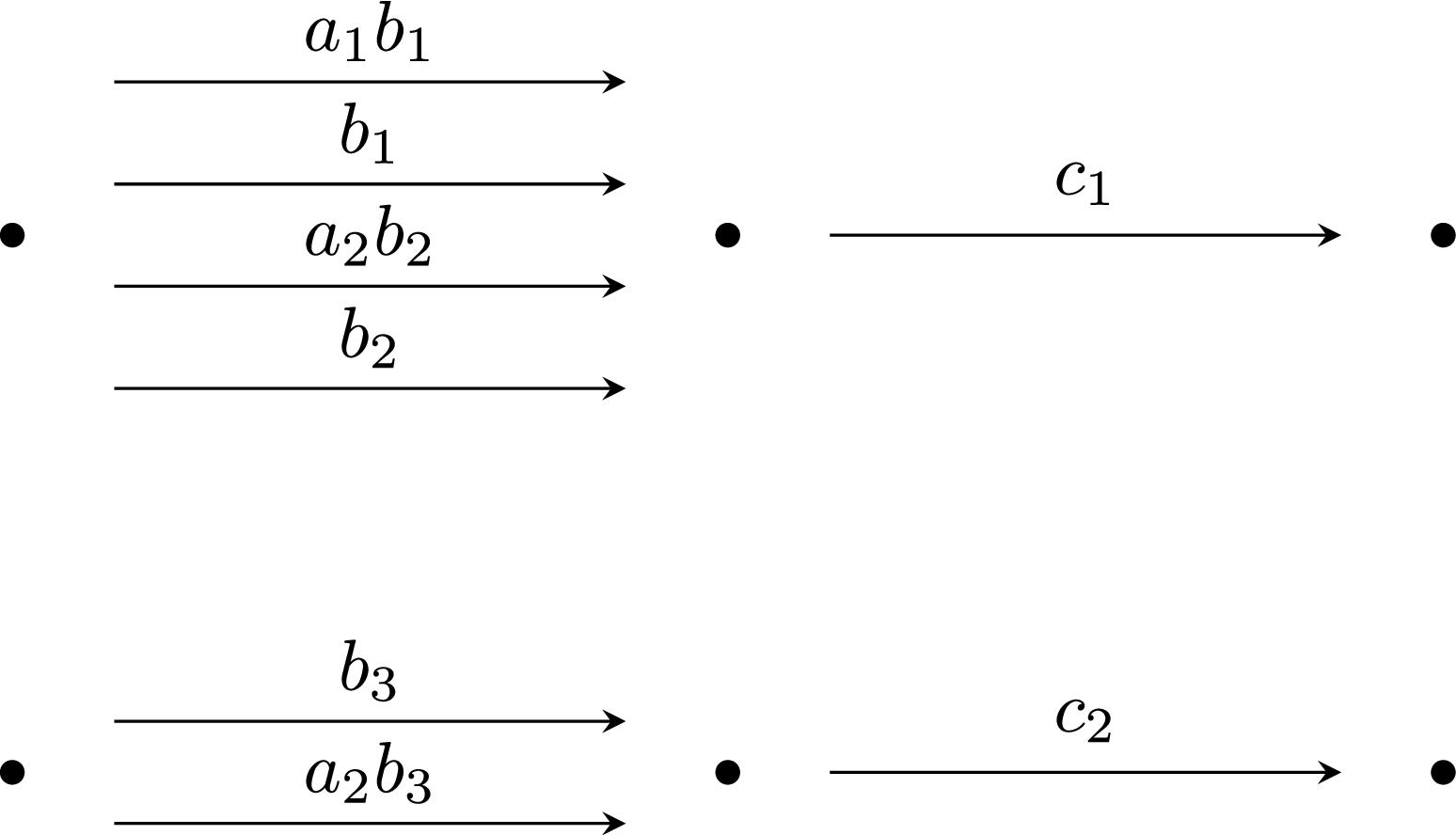}
						\caption{Quiver of length 2}
						\label{example_Q3}
        					\end{center}
        				\end{figure}
			\begin{Prop}
				The subspace $\n'\coloneqq\n_{Q'}=\mathrm{span}\Path(Q')$ is an ideal of $\n_Q$, and of $(m-1)$-step nilpotent. 
			\end{Prop}
			\begin{proof}
				It is easy to see that $\n'$ is an ideal, since 
				\[
					[\n_Q, \n_Q]\subset \mathrm{span}\{x\in\Path(Q)\mid \text{length of }x\geq 2 \}\subset \n'. 
				\]
				Hence, we have only to show that $\n'$ is of $(m-1)$-step nilpotent. Take $a\in S$. By definition, there exists $x\in \Path(Q)$ such that $ax$ is a path of length $m$. Note that $x\notin S$. We thus have $x\in \Path(Q')$ and is of length $m-1$. This yields that $\n'$ is of at least $(m-1)$-step nilpotent. 
				
				It remains to show that $\n'$ is of at most $(m-1)$-step nilpotent. Assume that $\n'$ is of $m$-step nilpotent. Then there exist $\beta_1, \beta_2,\ldots, \beta_m\in E'$ such that $\beta_1\beta_2\cdots\beta_m$ is a path of length $m$. By the definition of $S$, one has $\beta_1\in S$, This contradicts to $\beta_1\in E'$, which completes the proof. 
			\end{proof}
			Recall that the automorphism group $\mathrm{Aut}(Q)$ acts on $\Path(Q)$. By the definition of the action, the length of a path is preserved by $\mathrm{Aut}(Q)$. Recall also that an automorphism is a bijection $f: Q\to Q$ satisfying that $t(\alpha)=s(\beta)$ is equivalent to $t(f(\alpha))=t(f(\beta))$ for any $\alpha, \beta\in E$. Note that the same equivalence holds for any $\alpha, \beta\in\Path(Q)$. We here prove some properties about automorphisms of quivers regarding $S$ and $\Path(Q')$. 
			\begin{Prop}\label{S}
				Let $Q=(V, E, s, t)$ be a quiver of length $m$, and consider $S$ and $Q'$ defined above. Then every automorphism $f\in \mathrm{Aut}(Q)$ of the quiver $Q$ satisfies the following:  
				\begin{enumerate}[(1)]
					\item $f(S)=S$, 
					\item $f(\Path(Q'))=\Path(Q')$, 
					\item $f|_{\Path(Q')}\in\mathrm{Aut}(Q')$. 
				\end{enumerate}
			\end{Prop}
			\begin{proof}
				First of all, we prove (1). It is enough to show that $f(S)\subset S$. Take any $a\in S$. By definition, there exists $y\in\Path(Q)$ such that $ay$ is a path of length $m$. Since the path $f(ay)=f(a)f(y)$ is of length $m$, it satisfies $f(a)\in S$. The assertion (2) easily follows from (1), since $f$ is bijective and $\Path(Q)=S\sqcup\Path(Q')$. Finally, we prove (3). Note that $f|_{\Path(Q')}$ is bijective by (2). Hence, we have only to show that $t'(\alpha)=s'(\beta)$ is equivalent to $t'(f(\alpha))=s'(f(\beta))$ for $\alpha, \beta\in E'$. Note that $s'=s|_{E'}$ and $t'=t|_{E'}$. Therefore, the claim follows from the fact that $E'\subset \Path(Q)$ and $f$ is an automorphism of $Q$.  
			\end{proof}
			Note that an automorphism of $Q'$ cannot be extended to an automorphism of $Q$ in general. For example, in the quiver $Q'$ described in Example \ref{Ex_Q}, the edge $b_1$ can be mapped to $a_1b_1$ by an automorphism of $Q'$, but not by an automorphism of $Q$ (since they have different lengths in Q). The next proposition generalizes this example. 
			\begin{Prop}\label{Aut}
				Let $Q=(V, E, s, t)$, $S$, and $Q'$ be as above. Take $a\in S$ and $x\in\Path(Q')$ with $t(a)=s(x)$. Then there exists $f\in\mathrm{Aut}(Q')$ such that $f(x)=ax$. 
			\end{Prop}
			\begin{proof}
				Since $x\in\Path(Q')$,  there exist $x_1, x_2, \ldots, x_r\in E'$ such that $x=x_1x_2\cdots x_r$. Here let us define a map $f: E'\to E'$ by
				\[
					f(ax_1)=x_1,~~ f(x_1)=ax_1,~~f|_{E'\setminus \{x_1, ax_1\}}=\id.
				\]
				Then it is easy to see that $f$ is bijective and satisfies
				\[
					f(x)=f(x_1)f(x_2)\cdots f(x_r)=(ax_1)x_2\cdots x_r=ax. 
				\] 
				We have only to show $f\in \mathrm{Aut}(Q')$. Take any $\alpha, \beta\in E'$, and we claim that $t(\alpha)=s(\beta)$ is equivalent to $t(f(\alpha))=s(f(\beta))$.  The proof is divided into two cases. 
				
				The first case is $\beta\in \{x_1, ax_1\}$. In this case, the claim holds in a trivial sense, that is, there are no $\alpha\in E'$ such that $t(\alpha)=s(\beta)$ or $t(f(\alpha))=s(f(\beta))$. This can be proved as follows. 
				
				First of all, we show that there is no $\alpha\in E'$ such that $t(\alpha)=s(\beta)$. Assume that such $\alpha\in E'$ exists. Note that, since  $a\in S$, there exists $y\in\Path(Q)$ such that $ay$ is a path of length $m$. In the case of $\beta=x_1$, we have $t(\alpha)=s(\beta)=s(x_1)=t(a)=s(y)$. This yields that $\alpha y$ is a path of length $m$, which contradicts to $\alpha\in E'$. Also, in the case of $\beta=ax_1$, we have $t(\alpha)=s(ax_1)=s(ay)$. This is a contradiction, since $\alpha ay$ is a path of length $m+1$. 
				
				We next show that there is no $\alpha\in E'$ such that $t(f(\alpha))=s(f(\beta))$. This can be proved by the same argument, since $f(\beta)\in\{x_1, ax_1\}$. This completes the proof the first case. 
							
				The second case is $\beta\in E'\setminus\{x_1, ax_1\}$. In this case, one has $f(\beta)=\beta$, and hence $s(f(\beta))=s(\beta)$. On the other hand, it satisfies $t(\alpha)=t(f(\alpha))$ by definition, since $f$ only interchanges $x_1$ and $ax_1$. Therefore the claim also holds in this case, which completes the proof. 
			\end{proof}
			We then consider the nilpotent Lie algebra $\n_Q$ and the subalgebra $\n'$. The next proposition gives a sufficient condition for a derivation of $\n'$ to be extended to a derivation of $\n_Q$. The automorphism given in the above proposition plays a key role. 
			\begin{Prop}\label{Der}
				Assume that $D \in \Der(\n')$ satisfies: 
				\begin{itemize}
					\item $D$ is diagonal with respect to the basis $\Path(Q')$, 
					\item $D\circ f=f\circ D$ holds for every $f\in\mathrm{Aut}(Q')$.  
				\end{itemize}
				Then, the extension $\overline{D}$ of $D$ by $\overline{D}|_S=0$ is a derivation of $\n_Q$. 
			\end{Prop}
			\begin{proof}
				Take any $x, y\in\Path(Q)$, and we show that 
				\[
					\overline{D}[x, y]=[\overline{D}x, y]+[x, \overline{D}y]. 
				\]
				In the case of $x, y \in\Path(Q')$, the claim holds by $\overline{D}|_{\n'}\in\Der(\n')$. In the case of $x, y\in S$, the claim follows from $[x,y]=0$. It remains to study the case of $x\in S$ and $y\in\Path(Q')$. Since $\n'$ is an ideal of $\n_Q$, one has
				\[
					\overline{D}[x,y]=D[x,y]. 
				\]
				The right hand side of the claim satisfies
				\[
					[\overline{D}x,y]+[x,\overline{D}y]=[x,Dy],
				\]
				since $\overline{D}x=0$ and $\overline{D}y=Dy$. By the first hypothesis of the proposition, that is $D$ being diagonal, there exists $\lambda_y\in\mathbb{R}$ such that $Dy=\lambda_yy$. Therefore, the right hand side of the claim coincides with $[x,Dy]=[x,\lambda_yy]=\lambda_y[x,y]$. Therefore, we have only to show that 
				\[
					D[x,y]=\lambda_y[x,y]. 
				\]
				In the case that $[x, y]=0$, it is easy to see the claim holds. In the case that $[x, y]=xy$, there exists $f\in\mathrm{Aut}(Q')$ such that $f(y)=xy$, by Proposition \ref{Aut}. Ir then follows from the second hypothesis that 
				\[
					D[x,y]=D(xy)=D\circ f(y)= f\circ Dy=f(\lambda_yy)=\lambda_yf(y)=\lambda_y[x,y]
				\]
				which completes the proof. 
			\end{proof}
		\subsection{Ricci curvature}
			As in the previous subsection, let $Q$ be a finite quiver without cycles, $\Path(Q)$ be the set of all paths in $Q$, and $\n_Q$ be the  nilpotent Lie algebra obtained by $Q$. Let us consider an inner product $\langle, \rangle$ on $\n_Q$ which makes $\Path(Q)$ orthogonal. Denote by $\overline{x}\coloneqq x/|x|$, and then the basis $\{\overline{x}\mid x\in\Path(Q)\}$ is orthonormal. It follows from Equation (\ref{Ric_nice}) that 
			\begin{equation}\label{Ric}
            			\Ric(\overline{x})=\left(-\frac{1}{2}\left(\sum_{t(a)=s(x)}\langle[\overline{a}, \overline{x}], \overline{ax}\rangle^2+\sum_{t(x)=s(a)}\langle[\overline{x}, \overline{a}], \overline{xa}\rangle^2\right)+\frac{1}{2}\sum_{bc=x}\langle[\overline{b}, \overline{c}], \overline{x}\rangle^2\right)\overline{x}. 
            		\end{equation}
			We use the notations $S$, $Q'$ and $\n'=\n_{Q'}$ defined in Subsection \ref{subalgebra}. The aim of this subsection is to give an expression of $\Ric(\overline{x})$ in terms of the Ricci curvature $\Ric_{\n'}$ of $\n'$ with respect to $\langle, \rangle|_{\n'\times\n'}$. For this purpose, let us divide $\Path(Q')$ into the following three subsets:  
			\begin{align*}
				P_1&\coloneqq\{x\in\Path(Q')\mid \exists a\in S\text{ s.t. }t(a)=s(x)\},  \\
				P_2&\coloneqq\{ay\in\Path(Q')\mid a\in S, y\in P_1\}, \\
				P_0&\coloneqq\Path(Q')\setminus(P_1\cup P_2). 
			\end{align*}
			Note that these subsets are disjoint, since $ab$ is not a path for any $a$, $b\in S$. Therefore, $\Path(Q)$ can be decomposed into four disjoint subsets $S$, $P_1$, $P_2$, and $P_0$. 
			\begin{Prop}\label{RicSQ'}
				Under the above notations, we have the following:  
				\begin{enumerate}[(1)]
					\item $\Ric (\overline{x})=\left(\displaystyle-\frac{1}{2}\sum_{t(x)=s(y)}\langle [\overline{x},\overline{y}],\overline{xy}\rangle^2\right)\overline{x}$~~~~(for each $x\in S$)
					\item $\Ric (\overline{x})=\Ric_{\n'}(\overline{x})+\displaystyle\left(-\frac{1}{2}\sum_{\substack{t(a)=s(x)\\ a\in S}}\langle [\overline{a}, \overline{x}], \overline{ax}\rangle^2\right)\overline{x}$~~~~(for each $x\in P_1$)
					\item $\Ric (\overline{x})=\Ric_{\n'}(\overline{x})+\displaystyle\left(\frac{1}{2}\sum_{\substack{bc=x\\ b\in S}}\langle [\overline{b}, \overline{c}], \overline{x}\rangle^2\right)\overline{x}$~~~~(for each $x\in P_2$)
					\item $\Ric (\overline{x})=\Ric_{\n'}(\overline{x})$~~~~(for each $x\in P_0$)
				\end{enumerate}
			\end{Prop}
			\begin{proof}
				The Ricci curvature $\Ric(\overline{x})$ can be calculated by using (\ref{Ric}). We study the four cases individually. 
				
				The first case is $x\in S$. In this case, there does not exist $a\in \Path(Q)$ such that $t(a)=s(x)$. Furthermore, since $x$ is of length 1, it cannot be expressed as $x=bc$. Therefore, the first and third sigmas in (\ref{Ric}) vanish, which proves the assertion (1). 
				
				In order to study the remaining cases, let $x\in \Path(Q')$. Note that $xa$ is not a path for any $a\in S$.  Then it follows from (\ref{Ric}) that 
            			\begin{equation} \label{RicQ'}
					\Ric(\overline{x})=\Ric_{\n'}(\overline{x})+\left(-\frac{1}{2}\sum_{\substack{t(a)=s(x) \\ a\in S}}\langle[\overline{a}, \overline{x}], \overline{ax}\rangle^2+\frac{1}{2}\sum_{\substack{bc=x \\ b\in S}}\langle[\overline{b}, \overline{c}], \overline{x}\rangle^2\right)\overline{x}. 
				\end{equation}
				The first sigma in (\ref{RicQ'}) is nonzero if and only if $x\in P_1$. Also, the second sigma in (\ref{RicQ'}) is nonzero if and only if $x\in P_2$. This completes the proof. In fact, in the case of $x\in P_1$, the second sigma in (\ref{RicQ'}) is zero since $x\notin P_2$, which proves the assertion (2). The assertion (3) and (4) can be proved similarly. 
			\end{proof}
		\subsection{Proof of the main theorem}
			The following is our main theorem, which states that $\n_Q$ always admits an algebraic Ricci soliton. Moreover, such inner product on $\n_Q$ can be chosen to be compatible with the structure of the quiver $Q$.
			\begin{Thm}
				Let $Q$ be a finite quiver without cycles, and $\n_Q$ be the  nilpotent Lie algebra obtained by $Q$. Then, there exists an inner product $\langle, \rangle$ on $\n_Q$ which satisfies the following three conditions: 
				\begin{enumerate}
            				\renewcommand{\labelenumi}{(\roman{enumi})}
            				\item $\langle,\rangle$ is algebraic Ricci soliton satisfying $\Ric=-\id+D$, where $D\in \Der(\n_Q)$ is diagonal with respect to $\Path(Q)$,  
            				\item $\Path(Q)$ is an orthogonal basis with respect to $\langle, \rangle$, 
					\item $\langle,\rangle$ is preserved by $\mathrm{Aut}(Q)$. 
            			\end{enumerate}
			\end{Thm}
			\begin{proof}
				It is proved by induction with respect to the number of steps. First of all, we consider the case that $\n_Q$ is of $1$-step nilpotent, that is abelian.  Let $\langle, \rangle$ be  an inner product on $\n_Q$ which makes $\Path(Q)$ orthonormal. Then it satisfies Conditions (ii) and (iii). Furthermore, since $\n_Q$ is abelian, we have 
				\[
					\Ric=0=-\id+\id,~~~ \id\in\Der(\n_Q). 
				\]
				Therefore, it also satisfies (i). 
				
				Assume that the assertion holds for the $(r-1)$-step case, and we prove the $r$-step case. Let $\n_Q$ be the nilpotent Lie algebra obtained by $Q=(V, E, s, t)$, which is of $r$-step nilpotent. We use the notions $S$, $Q'$ and $\n'=\n_{Q'}$ defined in Section \ref{subalgebra}. Recall that 
				\[
					S\subset E, ~~~~\Path(Q')=\Path(Q)\setminus S, 
				\]
				and $\n'$ is of $(r-1)$-step nilpotent. Hence, by induction hypothesis, there exists an inner product $\langle, \rangle'$ on $\n'$ satisfying (i), (ii) and (iii).  
				We define an inner product $\langle, \rangle$ on $\n_Q$ as follows:   
				\begin{itemize}
					\item $\langle, \rangle|_{\n'\times\n'}=\langle, \rangle'$, 
            				\item $\Path(Q)$ is an orthogonal basis of $\n_Q$, 
            				\item $\langle a, a \rangle=1/2(\#\{x\in\Path(Q')\mid t(a)=s(x)\}+\#\{b\in S\mid t(b)=t(a)\}+1)$ for each $a\in S$. 
				\end{itemize}
				Note that $\langle, \rangle$ is positive definite, since $\langle a, a \rangle>0$ for each $a\in S$. We will prove that this inner product $\langle, \rangle$ satisfies Conditions (i), (ii) and (iii). By definition, it is easy to see that (ii) is satisfied. 
				
				One can also easily show that (iii) is satisfied. Take any $f\in\mathrm{Aut}(Q)$ and $x\in \Path(Q)$, and we prove that $\langle x, x\rangle=\langle f(x), f(x)\rangle$. It follows from Proposition \ref{S} that $f$ preserves $S$ and $\Path(Q')$, respectively. In the case of  $x\in\Path(Q')$, the claim follows from the induction hypothesis. Namely, since $\langle, \rangle'$ is preserved by the automorphism $f|_{\Path(Q')}$ of $Q'$, it satisfies 
				\[
					 \langle x, x\rangle=\langle x, x\rangle'=\langle (f|_{\Path(Q')})(x), (f|_{\Path(Q')})(x)\rangle'=\langle f(x), f(x)\rangle. 
				\]
				In the case of $x\in S$, since $\langle x, x\rangle$ is defined by information of $Q$, it is preserved by the automorphism $f$ of $Q$. 
								
				It remains to prove (i). First of all, we show that 
				\begin{equation}\label{norm}
					\langle [\overline{a}, \overline{x}] , \overline{ax} \rangle=\frac{1}{|a|}~~(\text{for }a\in S, x\in P_1, ax\in P_2). 
				\end{equation}
				Since $[a, x]=ax$ by definition, we have 
				\begin{equation}\label{norm2}
            				[\overline{a},\overline{x}]=\frac{1}{|a||x|}[a,x]=\frac{1}{|a||x|}ax=\frac{|ax|}{|a||x|}\overline{ax}.
            			\end{equation}
				Since $ax$ is a path, Proposition \ref{Aut} yields that there exists $f\in \mathrm{Aut}(Q')$ such that $f(x)=ax$. One knows that $f$ preserves $\langle,\rangle'$ by induction hypothesis. Hence we have $\langle x, x\rangle=\langle ax, ax\rangle$. By substituting this into (\ref{norm2}), one can complete the proof of Claim (\ref{norm}). 
				
				We next calculate $|a|$ for $a\in S$. Recall that $|a|$ is defined by using the numbers of certain vertices. Here we define some subsets of $S$, $P_1$, and $P_2$. Let us consider the set of target points of $S$, which is a finite set: 
				\[
					V_S\coloneqq\{t(a)\mid a\in S\}=\{v_1, v_2, \ldots, v_m\}. 
				\] 
				For each $j\in\{1, \ldots, m\}$, we define
				\begin{align*}
					S_j&\coloneqq\{a\in S\mid t(a)=v_j\}, \\
					P_1^j&\coloneqq\{x\in\Path(Q')\mid s(x)=v_j\}, \\
					P_2^j&\coloneqq\{ay\in\Path(Q')\mid a\in S_j, y\in P_1^j\}.  
				\end{align*}
				Namely, they are the set of paths whose targets are $v_j$, the set of paths whose sources are $v_j$, and the set of paths which pass through $v_j$, respectively. It is easy to see that they give partitions, that is, 
				\[
					\textstyle{S=\bigsqcup_{j=1}^m S_j, ~~P_1=\bigsqcup_{j=1}^m P_1^j, ~~P_2=\bigsqcup_{j=1}^m P_2^j}. 
				\]
				By the definition of the inner product, $\langle a,a\rangle$ is constant on each $S_j$. To be exact, every $a\in S_j$ satisfies $t(a)=v_j$ and hence 
				\begin{equation}\label{N_j}
					N_j\coloneqq\langle a, a\rangle=\frac{1}{2}(\# P_1^j+\# S_j+1). 
				\end{equation}
				We now calculate $\Ric(\overline{x})$ for each $x\in \Path(Q)$. The first case is $x\in S_j$ In this case, it follows Proposition \ref{RicSQ'} that 
				\[
					\Ric (\overline{x})=\biggr(-\frac{1}{2}\sum_{t(x)=s(y)}\langle[\overline{x}, \overline{y}], \overline{xy}\rangle^2\biggr)\overline{x}=\biggr(-\frac{1}{2N_j}\# P_1^j\biggr)\overline{x}~~~(x\in S_j). 
				\]
				In fact, since $x\in S_j$ and $t(x)=s(y)$, it satisfies $y\in P_1^j$ and $xy\in P_2^j$. Therefore, Equation (\ref{norm}) yields that 
				\[
					\langle[\overline{x},\overline{y}],\overline{xy}\rangle^2=1/|x|^2=1/N_j. 
				\]
				The remaining cases can be calculated similarly. In the case of $x\in P_1^j$, it satisfies 
				\[
					\Ric (\overline{x})-\Ric_{\n'}(\overline{x})=\biggr(-\frac{1}{2}\sum_{\substack{t(a)=s(x) \\ a\in S}}\langle[\overline{a}, \overline{x}], \overline{ax}\rangle^2\biggr)\overline{x}=\left(-\frac{1}{2N_j}\#S_j\right)\overline{x}~~~(x\in P_1^j). 
				\]
				In the case of $x\in P_2^j$, since $x$ can be written as $x=bc$ with $b\in S$ uniquely (and it satisfies $b\in S_j$), one has 
				\[
					\Ric (\overline{x})-\Ric_{\n'}(\overline{x})=\biggr(\frac{1}{2}\sum_{\substack{bc=x\\ b\in S}}\langle [\overline{b}, \overline{c}], \overline{x}\rangle^2\biggr)\overline{x}=\frac{1}{2N_j}\overline{x}~~~(x\in P_2^j).
				\]
				In the case of $\in P_0$, one knows $\Ric(\overline{x})=\Ric_{\n'}(\overline{x})$. Therefore, we have obtained a description of $\Ric(\overline{x})$ in terms of $\Ric(\overline{x})$, $N_j$, $S_j$, and $P_1^j$. 
				
				We have to show that $\Ric=-\id+D$, where $D\in\Der(\n_Q)$ is diagonal with respect to $\Path(Q)$. By induction hypothesis, it satisfies $\Ric_{\n'}=-\id_{\n'}+D'$, where $D'\in\Der(\n')$ is diagonal with respect to $\Path(Q')$. Hence, we have an expression  
            			\[
            				\Ric=-\id +\overline{D'}+A
            			\]
				where $\overline{D'}$ is the extension of $D'$ as in Proposition \ref{Der}, and
            			\begin{equation}\label{A}
            				Ax=
            				\begin{cases}
            					\left(-\frac{1}{2N_j}\# P_1^j+1\right)x &(x\in S_j), \\
            					\left(-\frac{1}{2N_j}\#S_j\right)x & (x\in P_1^j), \\
            					\frac{1}{2N_j}x & (x\in P_2^j), \\
            					0 & (x\in P_0). 
            				\end{cases}
            			\end{equation}
				By definition, $\overline{D'}+A$ is diagonal with respect to $\Path(Q)$. Furthermore, it satisfies $\overline{D'}\in\Der(\n_Q)$ by Proposition \ref{Der}. In order to show this, take any $f\in\mathrm{Aut}(Q')$. Since $f$ preserves $\langle, \rangle'$ and $[,]$, respectively, it commutes with $\Ric_{\n'}$. Therefore, $f$ commutes with $D'=\Ric_{\n}+\id_{\n'}'$, and hence $D'$ satisfies the assumptions in Proposition \ref{Der}. 
				
				It remains to show that $A\in\Der(\n_Q)$, that is
				\begin{equation}\label{ADer}
					A[x,y]=[Ax, y]+[x, Ay]~~(\text{for all } x,y\in\Path(Q)). 
				\end{equation}
				This is satisfied for $x, y\in\Path(Q)$ with $[x, y]=0$, since $A$ is diagonal with respect to $\Path(Q)$. Nonzero brackets are given as follows:  
				\begin{enumerate}
					\renewcommand{\labelenumi}{(\arabic{enumi})}
					\item $[P_0, S_j]\subset \mathrm{span}~(S_j),~[P_0, P_1^j]\subset\mathrm{span}~(P_1^j),~[P_0, P_2^j]\subset\mathrm{span}~(P_2^j),~[P_0, P_0]\subset\mathrm{span}~(P_0)$, 
					\item $[S_j, P_1^j]=\mathrm{span}~(P_2^j)$. 
				\end{enumerate}
				One can easily check Case (1), since $A|_{P_0}=0$ and $ad_{P_0}$ preserves the other subspaces. In order to show Case (2), let us take $a\in S_j$ and $z_1\in P_1^j$. By definition of $A$, one has
				\begin{align*}
					[Aa, z_1]+[a, Az_1]&=\left(-\frac{1}{2N_j}\# P_1^j+1\right)az_1+\left(-\frac{1}{2N_j}\#S_j\right)az_1 \\
					&=\left(-\frac{1}{2N_j}\left(\# P_1^j+\# S_j-2N_j\right)\right)az_1, \\
					A[a, z_1]&=\frac{1}{2N_j}az_1. 
				\end{align*}
				They coincide, since it satisfied $2N_j=\#P_i^j+\#S_j+1$ by the definition of $N_j$ (see (\ref{N_j})). This proves $A\in\Der(\n_Q)$. This completes the proof of the assertion (i), and hence the inner product $\langle, \rangle$ on $\n_Q$ satisfies all the desired conditions. 
        			\end{proof}

\end{document}